\documentclass[10pt]{article}
\usepackage{setspace}
\usepackage{tikz}
\usetikzlibrary{arrows}
\usepackage{amsthm,amsfonts,amssymb,epsfig,graphics,amsmath,amsbsy,subfigure}
\usepackage{todonotes}

\usepackage{enumitem}
\usepackage[colorlinks]{hyperref}
\usetikzlibrary{arrows,shapes,positioning}
\usetikzlibrary{decorations.markings}
\tikzstyle arrowstyle=[scale=1]
\tikzstyle directed=[postaction={decorate,decoration={markings,
    mark=at position .65 with {\arrow[arrowstyle]{stealth}}}}]
\usepackage{color}		
\usepackage{epsfig}

\newcommand{\intR}{\int\limits_{\mathbb{R}} }
\newcommand{\intRR}{\int\limits_{\mathbb{R}^2} }
\newcommand{\intL}{\int\limits }
\newcommand{\half}{^\infty_0 }
\newcommand{\RRn}{\RR^{n-1}\times\RR^{n-1}}
\newcommand{\RR}{\mathbb{R}} 
\newcommand{\xx}{\mathbf{x}} 
\newcommand{\yy}{\mathbf{y}} 
\newcommand{\uu}{\mathbf{u}} 
\newcommand{\vv}{\mathbf{v}} 
 
\newcommand{\xxi}{\boldsymbol{\xi}} 
\newcommand{\eeta}{\boldsymbol{\eta}} 
\newcommand{\aaa}{\mathbf{a}} 
\usepackage{amssymb}
\usepackage{amsmath}
\usepackage{graphicx}
\newtheorem{thm}{Theorem}
\newtheorem{cor}{Corollary}

\newtheorem{prop}{Proposition}

\theoremstyle{definition}

\newtheorem{rmk}{Remark}

\allowdisplaybreaks

\usepackage{bbding}
\newcommand{\eor}{\hfill {\scriptsize\OrnamentDiamondSolid}}

\title{On the determination of a function from its cone transform with fixed central axis}
\author{Sunghwan Moon\footnote{shmoon@unist.ac.kr}}
\date{{\em{Department of Mathematical Sciences,}}\\
{\em{Ulsan National Institute of Science and Technology,}}\\
{\em{Ulsan 689-798, Republic of Korea}}}

\begin{document}\maketitle
\begin{abstract}
A Radon-type transform called a \textbf{cone transform} that assigns to a given function its integral over various sets of the cones has arisen in the last decade in the context of the study of Compton cameras used in Single Photon Emission Computed Tomography.
Here, we study the cone transform for which the central axis of the cones of integration is fixed.
We present many of its properties, such as two inversion formulas, a stability estimate, and uniqueness and reconstruction for a local data problem.
\end{abstract}

\section{Introduction}
\textbf{Single Photon Emission Computed Tomography (SPECT)}, a useful medical diagnostic tool, inspects internal organs and produces pictures of their internal processing using the distribution of an isotope.
SPECT typically provides the information as cross-sectional slice but it is easy to reformat or manipulate into other types of images.
To obtain the image in SPECT, a gamma-emitting radioisotope is injected into the patient, usually, via the bloodstream.
This radioisotope passes through the body and is detected by the scan.

When the conventional gamma camera is used in SPECT, it is common to eliminate large parts of the signal since it is mechanically collimated, which means counts only the particles coming from a narrow cone of directions and discards the rest. 
Because of such low efficiency, a Compton Camera was introduced for use in SPECT.
A Compton camera has very high sensitivity and flexibility of geometrical design, so it has attracted a lot of interest in many areas, including monitoring nuclear power plants and in astronomy.
%

A typical Compton camera consists of two planar detectors: a scatter detector and an absorption detector, positioned one behind the other. 
A photon emitted in the direction of the camera undergoes Compton scattering in the scatter detector positioned ahead, and is absorbed in the absorption detector (see Figure~\ref{fig:compton}). 
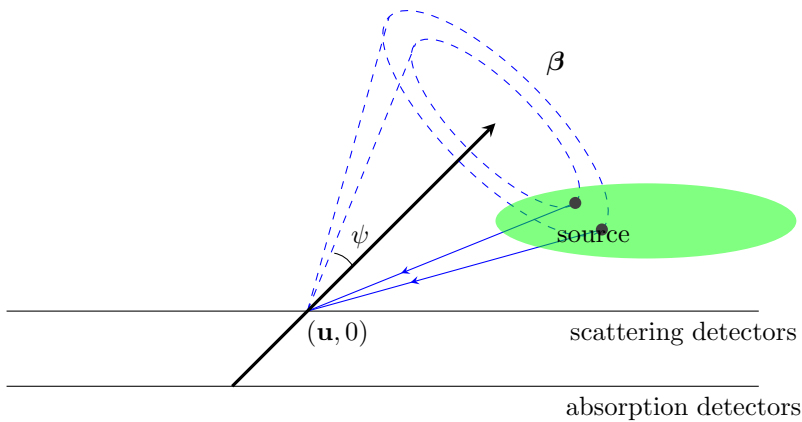
\begin{figure}
\begin{center}
  \begin{tikzpicture}[>=stealth]
   \draw (-3,0) -- (7,0) ;
      \draw (-3,1) -- (7,1) ;
       \draw[blue,directed] (3.5+2/1.4141,3.5-2/1.4141) -- (1,1) ;
   \draw[blue,directed] (3.5+1.5/1.4141,3.5-1.5/1.4141) -- (1,1) ;
   \draw[blue,dashed] (1,1) -- (3.5-1.5/1.4141,3.5+1.5/1.4141) ;
   \draw[blue,dashed] (1,1) -- (3.5-2/1.4141,3.5+2/1.4141) ;
   \draw[very thick,->] (0,0) -- (3.5,3.5) ;
   \draw[blue,rotate around={-45:(3.5,3.5)},dashed] (5.5,3.5) arc (0:360:2cm and 0.7cm);
   \draw[blue,rotate around={-45:(3.5,3.5)},dashed] (5,3.5) arc (0:360:1.5cm and 0.5cm);
  \fill[semitransparent,green] (5.5,2.2) ellipse (2 and 0.5) ;5
   \fill[gray!50!black]  (3.5+1.5/1.4141,3.5-1.5/1.4141) circle (0.08);
      \fill[gray!50!black]  (3.5+2/1.4141,3.5-2/1.41411) circle (0.08);
   \draw (1.6,1.6) arc (30:80:10pt);

    \node at (6,0.7) {scattering detectors};
    \node at (6,-0.3) {absorption detectors};
   \node at (1.4,0.7) {$(\uu,0)$};
   \node at (4.3,4.3) {$\boldsymbol\beta$};
   \node at (1.7,2) {$\psi$};
   \node at (4.8,2) {source};
  \end{tikzpicture}
  \caption{Schematic representation of a Compton camera}
\end{center}
\label{fig:compton}
\end{figure}
In each detector, the position of the hit and energy of the photon are measured.
A difference vector between two device positions determines the central axis of a cone.
The scattering angle $\psi$ from the central axis can be computed from the measured energies and electron mass as follows:
$$
\cos\psi=1-\frac{m c^2 \Delta E}{(E-\Delta E)E},
$$
where $m$ is the mass of the electron, $c$ is the speed of light, $E$ is the initial gamma ray energy, and $\Delta E$ is the energy transferred to the electron in the scattering process~\cite{allmarasdhkk12,maximfp09}. 
Therefore, we get the integral of the distribution of the radiation source over cones with central axis ${\beta}$, vertex $u$ at the position of the scatter detector and scattering angle $\psi$. 
We reserve the name \textbf{cone transform} for the surface integrals of a source distribution over a family of cones.

Many inversion formulas for various types of cone transforms were derived~\cite{baskozg98,gouiaa14,gouia14,smith05,maximfp09}. 
In particular, the cone transform with a fixed central axis was studied in~\cite{creeb94,haltmeier14,moonct14,nguyentg05,truongnz07}. 
Cree and Bones derived the inversion formula for the cone transform in  \cite{creeb94} and Nguyen, Truong, and Grangeat obtained another inversion formula in \cite{nguyentg05}.
Haltmeier first defined an $n$-dimensional cone transform and found the inversion formula in \cite{haltmeier14}.
Jung and Moon discussed a relation between two existing formulas: one derived from Cree and Bones and the other from Nguyen, Truong, and Grangeat. 
In the same paper, they also obtained stability estimates for a more general form of the 3-dimensional cone transform in \cite{moonct14}.
In this article, we also deal this $n$-dimensional cone transform. 

A 2-dimensional cone transform becomes a $V$-line Radon transform, which integrates a function along coupled rays with a common vertex. 
This $V$-line Radon transform has been studied in the context of Single Scattering Optical Tomography~\cite{florescums09,florescums10,florescums11}.
Many works~\cite{allmarasdhkk12,ambartsoumian12,moonvrt13,morvidonentz10,truongn11} derived inversion formulas for various versions of the $V$-line Radon transform.

The definition of the cone transform is formulated precisely in section~\ref{sec:formulation}. 
Section~\ref{sec:properties} is devoted to elementary properties of the cone transform including an analog of the Fourier slice theorem. 
Two inversion formulas are presented in section~\ref{sec:inversion}. 
We describe the range of the cone transform in odd dimensions in section~\ref{sec:range}.
In section~\ref{sec:isometry}, we show that taking a linear operator on the cone transform is an isometry and discuss a stability estimate.
In section~\ref{sec:uniqueness}, uniqueness and reconstruction for a partial data problem are studied.
\section{Definition}\label{sec:formulation}
Let $\mathbf f$ be a function on $\RR^3$ with compact support in the upper half space $\RR^2\times[0,\infty)$.
We define the cone transform by
$$
\mathcal C\mathbf f(\uu,s):=\displaystyle\intL^{2\pi}_0\intL\half\mathbf f(\uu+zs\boldsymbol\theta,z) z  d\theta dz,
$$ 
for $(\uu,s)=(u_1,u_2,s)\in\RR^{2}\times[0,\infty)$ (see Figure~\ref{fig:integrationdomain1} (a)).
Here $\boldsymbol\theta=(\cos\theta,\sin\theta)\in S^1$.
Let the function $f$ on $\RR^4$ satisfy $f(\xx,\yy)=\mathbf f(\xx,|\yy|)$ for $(\xx,\yy)=(x_1,x_2,y_1,y_2)\in\RR^2\times\RR^2$ and the cone transform of $f$ be defined by
$$
Cf(\uu,\vv)=\displaystyle \intRR f(\uu+|\vv|\yy,\yy) d\yy.
$$
Then we have $Cf(\uu,\vv)=\mathcal C\mathbf f(\uu,|\vv|)$.
In fact, making a change of the variables gives
\begin{equation*}
\begin{split}
Cf(\uu,\vv)&=\displaystyle \intRR f(\uu+|\vv|\yy,y_1,y_2) d\yy=\displaystyle \intL^{2\pi}_0 \intL\half \mathbf f(\uu+|\vv|r\boldsymbol\theta,r)r drd\theta\\
&=\mathcal C\mathbf f(\uu,|\vv|).
\end{split}
\end{equation*}

Let us consider a natural $n$-dimensional analog of the cone transform for the function $\mathbf f$ on $\RR^n$ with compact support in the upper half space $\RR^{n-1}\times[0,\infty)$.
As in the 3-dimensional case, the cone transform is defined by
\begin{equation*}
\mathcal C\mathbf f(\uu,s):=\left\{\begin{array}{ll}\displaystyle\intL_{S^{n-2}}\intL\half\mathbf f(\uu+zs\boldsymbol\theta,z) z^{n-2}  dz dS(\boldsymbol\theta),&\mbox{ if }n\geq3,\\
\displaystyle\intL\half\mathbf f(\uu+zs,z)+\mathbf f(\uu-zs,z)  dz ,&\mbox{ if }n=2,\end{array}\right.
\end{equation*}
for $(\uu,s)=(u_1,u_2,\cdots,u_{n-1},s)\in\RR^{n-1}\times[0,\infty)$.
Here $dS(\boldsymbol\theta)$ is the standard measure on the unit sphere $S^{n-2}$.
When $n=2$, $\mathcal C\mathbf f$ is the $V$-line Radon transform whose integral domain is the set of $V$-shape lines (see Figure \ref{fig:integrationdomain1} (b)).
Similar to the definition of $Cf$ for the 3-dimensional case, we define $Cf$ for a function $f$ on $\RR^{2(n-1)}$ with $f(\xx,\yy)=\mathbf f(\xx,|\yy|)$ for $(\xx,\yy)\in\RR^{n-1}\times\RR^{n-1}$ by
\begin{equation}\label{eq:defiofcf}
Cf(\uu,\vv)=\displaystyle \intL_{\RR^{n-1}} f(\uu+|\vv|\yy,\yy) d\yy,\qquad\mbox{for } (\uu,\vv)\in\RR^{n-1}\times\RR^{n-1}.
\end{equation}
Again, we have $Cf(\uu,\vv)=\mathcal C\mathbf f(\uu,|\vv|)$.
Our goals are to reconstruct $f$ (or $\mathbf f$) from $Cf$ (or $\mathcal C\mathbf f$) and to study properties of this cone transform.
\begin{rmk}\label{rmk:2drelation}
The above definition~\eqref{eq:defiofcf} includes the $n=2$ case. 
When $n=2$, $Cf$ becomes an integral of $f$ along the line perpendicular to 
$$
(1,-v)/\sqrt{1+|v|^2}
$$
with signed distance
$$
u/\sqrt{1+|v|^2}\quad\mbox{(see Figure \ref{fig:integrationdomain1} (b))}.
$$
In this case the measure for the line becomes
$$
\sqrt{1+|v|^2}dy.
$$
Hence we have a relation between $Cf$ and $Rf$:
\begin{equation}\label{eq:relationbetweenrfandcf}
Cf(u,v)=\sqrt{1+|v|^2}Rf\left(\frac{(1,-v)}{\sqrt{1+|v|^2}},\frac u{\sqrt{1+|v|^2}}\right),
\end{equation}
where $Rf$ is the regular Radon transform defined by 
$$
Rf(\boldsymbol\omega,t)=\intR f(t\boldsymbol\omega+s\boldsymbol\omega^\perp)ds,\qquad\mbox{for } (\boldsymbol\omega,t)\in S^1\times\RR.
$$
Notice that 
$$
Rf\left(\frac{(1,-v)}{\sqrt{1+|v|^2}},\frac u{\sqrt{1+|v|^2}}\right)=Rf\left(\frac{(1,v)}{\sqrt{1+|v|^2}},\frac u{\sqrt{1+|v|^2}}\right).
$$
\eor\end{rmk}

\begin{figure}
\begin{center}
  \begin{tikzpicture}[>=stealth]
   \draw (-2,0) -- (4,0) ;
   \draw (4,0) -- (6,4) ;
   \draw (0,4) -- (6,4) ;
  \draw (-2,0) -- (0,4) ;
  \draw[thick] (2,2) -- (5,5) ;
   \draw[thick] (2,2) -- (-1,5) ;
    \draw[<->,loosely dashed] (3,3) -- (2,3);
      \draw[->,dashed] (2,2) -- (2,6) ;
   \draw[blue,dashed] (4,4) arc (0:180:2cm and 0.35cm);
      \draw[blue] (0,4) arc (180:360:2cm and 0.35cm);

  \node at (2,1.5) {$(\mathbf u,0)$};
    \coordinate (O1) at (6,0) ;
    	
    \draw[->] (8,-1.1) -- (8,6);
     \draw[->] (7.5,1) -- (12,1);
    \draw[thick] (9,1) -- (7,5);
    \draw[thick] (9,1) -- (11,5);
    \draw[->,loosely dashed] (9,1) -- (9,5.5);
    \draw[thick,densely dashed] (9,1) -- (8,-1);
    \draw[<->,loosely dashed] (10,3) -- (9,3);
      \node at (2.5,3.5) {$s$};
        \node at (9.5,3.5) {$s$};
  \node at (9,0.5) {$(u,0)$};
    \node at (12,0.3) {$x$};
        \node at (8.3,6.2) {$y$};
         \node at (2,-1) {(a)};
        \node at (9,-1) {(b)};
  \end{tikzpicture}
  \caption{Integration domain (a) a cone (b) a $V$-line and a line }
\end{center}
\label{fig:integrationdomain1}
\end{figure}
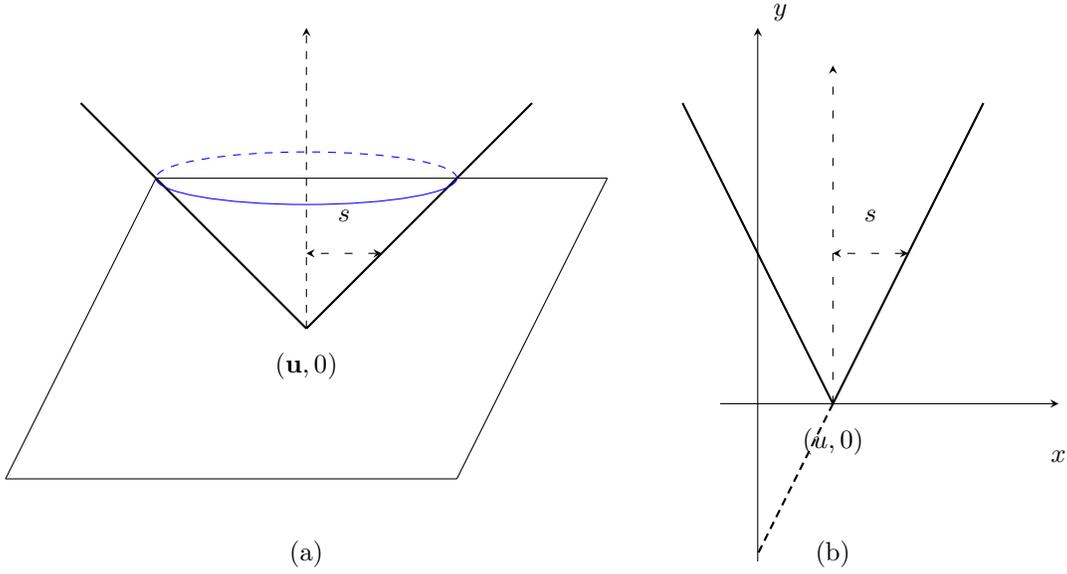

\section{Elementary properties}\label{sec:properties}
Let $\mathcal S(\RR^n)$ be the Schwartz class of infinitely differentiable functions $\mathbf f$ with $\sup_{\xx\in\RR^n}|\xx^\alpha\partial^\beta \mathbf f(\xx)|<\infty$ for any multiindices $\alpha$ and $\beta$.
We introduce 
\begin{equation*}
\begin{split}
\mathcal S_r(\RR^{n-1}\times\RR^{n-1})=\{f\in\mathcal S(\RR^{2(n-1)}):&f(\xx,\yy)=f(\xx,U\yy)\quad\mbox{ for any }(\xx,\yy)\in\RR^{n-1}\times\RR^{n-1}\\
&\mbox{and for all orthonormal transformations }U\}.
\end{split}
\end{equation*}

\begin{thm}\label{lem:fourierslicetheorem}
For $f\in \mathcal S_r(\RRn)$, we have
$$
\mathcal F_{1}(Cf)(\boldsymbol\xi,\vv)=\mathcal F_{} f(\boldsymbol\xi,|\vv|\boldsymbol\xi)=\mathcal F_{}f(\boldsymbol\xi,\vv|\boldsymbol\xi|),
$$
where $\mathcal F_{1}(Cf)$ and $\mathcal F_{}f$ are the $n-1$-dimensional and $2(n-1)$-dimensional Fourier transforms of $Cf$ and $f$ with respect to $\uu$ and $(\xx,\yy)$, respectively.
\end{thm}
\begin{proof}
Taking the $n-1$-dimensional Fourier transform of $Cf(\uu,\vv)$ with respect to $\uu$ yields
$$
\mathcal F_{1}(Cf)(\boldsymbol\xi,\vv)=\intL_{\RR^{n-1}}\mathcal F_{1}f(\boldsymbol\xi,\yy)e^{i|\vv|\yy\cdot\boldsymbol\xi}d\yy=\mathcal F_{} f(\boldsymbol\xi,-|\vv|\boldsymbol\xi),
$$
where $\mathcal F_{1}f$ is the $n-1$-dimensional Fourier transform of $f$ with respect to $\xx$.
Since the Fourier transform of a radial function is also radial, we have the assertion.
\end{proof}
Note that if $f$ is a radial function on $\RR^n$ and $f(\xx)=\mathbf f(|\xx|)$ , then 
\begin{equation}\label{eq:fourierofradialfunction}
\mathcal F_{{}}f(\boldsymbol\xi)=(2\pi)^\frac n2|\boldsymbol\xi|^\frac{2-n}2\mathbf H_{\frac{n-2}2}\mathbf f(|\boldsymbol\xi|),
\end{equation}
where $\mathcal F_{{}}f$ is the $n$-dimensional Fourier transform of $f$ and
$$
\mathbf H_{\frac{n-2}2}\mathbf f(\rho)=\intL\half \mathbf f(s)s^\frac n2J_{\frac{n-2}2}(s\rho)ds,
$$
where $J_k$ is the Bessel function of the first kind of order $k$.
\begin{cor}
Let $f(\xx,\yy)\in \mathcal S_r(\RRn)$ and $\mathbf f$ be a function on $\RR^{n-1}\times[0,\infty)$ with $\mathbf f(\xx,|\yy|)=f(\xx,\yy)$.
Then we have
\begin{equation}\label{eq:fourierslicetheorem3}
\mathcal F_{1}(\mathcal C\mathbf f)(\boldsymbol\xi,s)=(2\pi)^{\frac{n-1}2}|\boldsymbol\xi|^\frac{3-n}2s^\frac{3-n}2\mathbf H_{\frac{n-3}2}\mathcal F_{1}\mathbf f(\boldsymbol\xi,s|\boldsymbol\xi|).
\end{equation}
\end{cor}
%
\begin{rmk}
When $n=3$,~\eqref{eq:fourierslicetheorem3} was already derived in~\cite{creeb94}.
\eor\end{rmk}
\begin{prop}\label{thm:adjoint}
The cone transforms $C$ and $\mathcal C$ are self-adjoint, in the sense that for $f,g\in S_r(\RR^{n-1}\times\RR^{n-1})$ with $f(\xx,\yy)=\mathbf f(\xx,|\yy|)$ and $g(\uu,\vv)=\mathbf g(\uu,|\vv|)$,
\begin{equation}\label{eq:selfadjoint}
\displaystyle\intL_{\RR^{n-1}}\intL_{\RR^{n-1}}Cf(\uu,\vv)g(\uu,\vv)d\uu d\vv=\intL_{\RR^{n-1}}\intL_{\RR^{n-1}} f(\xx,\yy)C g(\xx,\yy)  d\xx d\yy
\end{equation}
and 
\begin{equation}\label{eq:selfadjoint2}
\displaystyle\intL_{\RR^{n-1}}\intL\half \mathcal C\mathbf f(\uu,s)\mathbf g(\uu,s)s^{n-2}ds d\uu=\intL_{\RR^{n-1}}\intL\half \mathbf f(\xx,z)\mathcal C \mathbf g(\xx,z) z^{n-2} dz d\xx.
\end{equation}
\end{prop}
\begin{proof}
We start out from
\begin{equation*}
\begin{split}
\displaystyle\intL_{\RR^{n-1}}\intL_{\RR^{n-1}}Cf(\uu,\vv)g(\uu,\vv)d\uu d\vv&\displaystyle=\intL_{\RR^{n-1}}\intL_{\RR^{n-1}}\intL_{\RR^{n-1}}f(\uu+|\vv|\yy,\yy)d\yy\; g(\uu,\vv)d\uu d\vv\\
&\displaystyle=\intL_{\RR^{n-1}}\intL_{\RR^{n-1}}f(\xx,\yy)\intL_{\RR^{n-1}} g(\xx-\yy|\vv|,\vv)d\vv d\yy d\xx\\
&\displaystyle=\intL_{\RR^{n-1}}\intL_{\RR^{n-1}}f(\xx,-\yy)\intL_{\RR^{n-1}} g(\xx+\yy|\vv|,\vv)d\vv d\yy d\xx,
\end{split}
\end{equation*}
where in the last line, we changed the variables $\yy\to-\yy$.
Since $f$ is a radial function in $\yy$, we have
\begin{equation*}
\displaystyle\intL_{\RR^{n-1}}\intL_{\RR^{n-1}}Cf(\uu,\vv)g(\uu,\vv)d\uu d\vv
\displaystyle=\intL_{\RR^{n-1}}\intL_{\RR^{n-1}}f(\xx,\yy)\intL_{\RR^{n-1}} g(\xx+\yy|\vv|,\vv)d\vv d\yy d\xx,
\end{equation*}
which is our assertion.
\end{proof}

\begin{prop}
For $f,g\in\mathcal S_r(\RRn),$ we have
$$
C(f*g)=Cf*Cg.
$$
\end{prop}
This theorem follows from Theorem~\ref{lem:fourierslicetheorem} and $\mathcal F_{{}}(f*g)=(\mathcal F_{{}}f)(\mathcal F_{{}}g)$.

By the definition of $Cf$, we notice that $Cf$ is the integral of $f$ over an $n-1$-dimensional plane. 
It is a natural idea that $Cf$ can be converted to the $2(n-1)$-dimensional regular Radon transform by an $n-2$-dimensional integration.
(Indeed, we showed that the 2-dimensional cone transform is converted to the 2-dimensional regular Radon transform in Remark~\ref{rmk:2drelation}.)
Let the regular Radon transform be defined by 
$$
Rf(\boldsymbol\omega,t)=\intL_{\boldsymbol\omega^\perp}f(t\boldsymbol\omega+\boldsymbol\tau)d\boldsymbol\tau\qquad\mbox{for } (\boldsymbol\omega,t)\in S^{n-1}\times\RR.
$$
Then this can be represented by
\begin{equation}\label{eq:changevariableofradon}
Rf(\boldsymbol\omega,t)=\sqrt{1+|\boldsymbol\omega'/\omega_1|^2}\intL_{\RR^{n-1}}f\left(-\frac{\boldsymbol\omega'\cdot\boldsymbol\tau}{\omega_1}+t\sqrt{1+|\boldsymbol\omega'/\omega_1|^2},\boldsymbol\tau\right)d\boldsymbol\tau
\end{equation}
for $\omega_1\neq0$.
Here $\boldsymbol\omega=(\omega_1,\omega_2,\cdots,\omega_{n-1})=(\omega_1,\boldsymbol\omega')\in S^{n-1}$ and $\boldsymbol\tau\in\RR^{n-1}$.

Now we convert the $4$-dimensional cone transform $Cf$ to the 4-dimensional regular Radon transform.
For $(a,b)\in\RR^2$, we integrate $Cf(au_2+b,u_2,\vv)$ with respect to $u_2$:
\begin{equation*}
\begin{split}
\intR Cf(au_2+b,u_2,\vv)du_2&=\intR \intRR f(au_2+b+|\vv|y_1,u_2+|\vv|y_2,\yy)d\yy du_2\\
&=\intRR \intR  f(au_2+|\vv|y_1-a|\vv|y_2+b,u_2,\yy)du_2d\yy,
\end{split}
\end{equation*}
which is equivalent to
\begin{equation}\label{eq:relation}
\begin{split}
\intR Cf(au_2+b,u_2,\vv)du_2=
((1+a^2)(1+|\vv|^2))^{-1/2}Rf\left(\frac{(-1,a,|\vv|,-a|\vv|)}{\sqrt{(1+a^2)(1+|\vv|^2)}},\frac{b}{\sqrt{(1+a^2)(1+|\vv|^2)}}\right).
\end{split}
\end{equation}
(see \eqref{eq:changevariableofradon}).
Since $f(\xx,\yy)$ is radial in $\yy$, we have for any $\boldsymbol\alpha=(\alpha_1,\alpha_2)\in\RR^2$,
$$
f(au_2+\alpha_1 y_1+\alpha_2 y_2+b,u_2,\yy)=f(au_2+\alpha_1 y_1+\alpha_2 y_2+b,u_2,U_{11}y_1+U_{12}y_2,U_{21}y_1+U_{22}y_2),
$$
where $U=(U_{\{i,j\}})$ is any 2$\times2$ orthogonal matrix.
Then, we have
\begin{equation}\label{eq:aux}
\begin{split}
&\intRR \intR  f(au_2+\alpha_1 y_1+\alpha_2 y_2+b,u_2,\yy)du_2d\yy\\=
&\intRR \intR  f(au_2+\alpha_1 y_1+\alpha_2 y_2+b,u_2,U_{11}y_1+U_{12}y_2,U_{21}y_1+U_{22}y_2)du_2d\yy\\=
&\intRR \intR  f(au_2+(\alpha_1U_{11}+\alpha_2U_{12}) y_1+(\alpha_1U_{21}+\alpha_2U_{22})y_2+b,u_2,\yy)du_2d\yy,
\end{split}
\end{equation}
where we changed the variables $U\yy\to\yy$.
%
Combining two equations~\eqref{eq:relation} and \eqref{eq:aux}, we have the following relation: for $|\boldsymbol\alpha|^2=|\vv|^2(1+a^2)$,
\begin{equation*}
\begin{split}
&\intR Cf(au_2+b,u_2,\vv)du_2\\
&=((1+a^2)(1+|\vv|^2))^{-1/2}Rf\left(\frac{(-1,a,\alpha_1,\alpha_2)}{\sqrt{(1+a^2)(1+|\vv|^2)}},\frac{b}{\sqrt{(1+a^2)(1+|\vv|^2)}}\right).
\end{split}
\end{equation*}
By a similar argument, we have the following theorem:
\begin{prop}\label{thm:relation}
Let $f\in S_r(\RRn)$ and $(\boldsymbol\alpha,\aaa,b)\in\RR^{n-1}\times\RR^{n-2}\times\RR$. Then we have for $|\boldsymbol\alpha|=|\vv|^2(1+|\aaa|^2)$ and $\uu'=(u_2,u_3,\cdots,u_{n-1})\in\RR^{n-2}$,
\begin{equation}\label{eq:relation1}
\begin{split}
&\intL_{\RR^{n-2}} Cf(\aaa\cdot \uu'+b,\uu',\vv)d\uu'\\
&=((1+|\aaa|^2)(1+|\vv|^2))^{-1/2}Rf\left(\frac{(-1,\aaa,\boldsymbol\alpha)}{\sqrt{(1+|\aaa|^2)(1+|\vv|^2)}},\frac{b}{\sqrt{(1+|\aaa|^2)(1+|\vv|^2)}}\right).
\end{split}
\end{equation}
\end{prop}
\begin{rmk}\label{rmk:inversion}
Taking an inversion formula for the regular Radon transform on~\eqref{eq:relation1}, $f$ can be recovered from $Cf$. 
\eor\end{rmk}
\section{Inversion formulas}\label{sec:inversion}
Although we already show how to recover $f$ from $Cf$ in Remark~\ref{rmk:inversion}, we present two explicit inversion formulas for the cone transform in this section.
 
For $k<n$ we define the linear operator $I^k$ by
$$
\mathcal F_{{}}(I^k f)(\boldsymbol\xi)=|\boldsymbol\xi|^{-k}\mathcal F_{{}}f(\boldsymbol\xi).
$$
The linear operator $I^k$ is called the Riesz potential.
For $f\in\mathcal S(\RR^n)$, $\mathcal F_{{}}(I^k f)\in L^1(\RR^n)$, hence $I^k f$ makes sense and $I^{-k}I^k f=f$.
When $I^k_1$ is applied to functions on $\RR^{n-1}\times\RR^{n-1}$, it acts on the first $n-1$-dimensional variable $\uu$ or $\xx$. 
Also, for $f\in\mathcal S_r(\RRn)$, we have $\mathcal F_{1} (Cf)(\boldsymbol\xi,\vv)=\mathcal F_{}f(\boldsymbol\xi,\boldsymbol\xi|\vv|)$ by Theorem~\ref{lem:fourierslicetheorem}, so for a fixed $\vv$, $\mathcal F_{1} (Cf)(\boldsymbol\xi,\vv)\in\mathcal S(\RR^{n-1})$ and therefore, $I^k_1 (Cf)$ makes sense.
\begin{thm}\label{thm:inversionseries}
Let $f\in S_r(\RR^{n-1}\times\RR^{n-1})$ with $f(\xx,\yy)=\mathbf f(\xx,|\yy|)$.
For $k<n-1$, we have
\begin{equation}\label{eq:inversionseries}
f=(2\pi)^{1-n}I_1^{-k}C I_1^{k+1-n}F,\qquad F=Cf,
\end{equation}
and
\begin{equation*}
\mathbf f=(2\pi)^{1-n}I_1^{-k}\mathcal C I_1^{k+1-n}\mathbf F,\qquad \mathbf F=\mathcal C\mathbf f.
\end{equation*}
\end{thm}
\begin{proof}
We start out from the Fourier inversion formula
\begin{equation*}
\begin{split}
I^k_1 f(\xx,\yy)&=\displaystyle(2\pi)^{2(1-n)}\intL_{\RR^{n-1}}\intL_{\RR^{n-1}}|\boldsymbol\xi|^{-k}\mathcal F_{}f(\boldsymbol\xi,\boldsymbol\eta)e^{i\xx\cdot\boldsymbol\xi}e^{i\yy\cdot\boldsymbol\eta}d\boldsymbol\xi d\boldsymbol\eta\\
&=\displaystyle(2\pi)^{2(1-n)}\intL_{\RR^{n-1}}\intL_{\RR^{n-1}}|\boldsymbol\xi|^{n-1-k}\mathcal F_{}f(\boldsymbol\xi,\vv|\boldsymbol\xi|)e^{i\xx\cdot\boldsymbol\xi}e^{i\yy\cdot\vv |\boldsymbol\xi|}d\boldsymbol\xi d\vv,
\end{split}
\end{equation*}
where in the last line, we changed the variables $\boldsymbol\eta\to\vv|\boldsymbol\xi|$.
Since $\mathcal F_{}f(\boldsymbol\xi,\boldsymbol\eta)$ and $I^k_1 f(\xx,\yy)$ are radial on $\boldsymbol\eta$ and $\yy$,
we obtain
$$
I^k_1 f(\xx,\yy)=\displaystyle(2\pi)^{2(1-n)}\intL_{\RR^{n-1}}\intL_{\RR^{n-1}}|\boldsymbol\xi|^{n-1-k}\mathcal F_{}f(\boldsymbol\xi,\boldsymbol\xi|\vv|)e^{i\xx\cdot\boldsymbol\xi}e^{i\vv\cdot\boldsymbol\xi|\yy|}d\boldsymbol\xi d\vv,
$$
which is equivalent to
$$
I^k_1 f(\xx,\yy)=\displaystyle(2\pi)^{2(1-n)}\intL_{\RR^{n-1}}\intL_{\RR^{n-1}}|\boldsymbol\xi|^{n-1-k}\mathcal F_{1}(Cf)(\boldsymbol\xi,\vv)e^{i(\xx+|\yy|\vv)\cdot\boldsymbol\xi}d\boldsymbol\xi d\vv.
$$
Here we used Theorem~\ref{lem:fourierslicetheorem}. 
Since the inner integral can be represented by the Riesz potential, we have
\begin{equation*}
\begin{split}
I^k_1 f(\xx,\yy)&=\displaystyle(2\pi)^{2(1-n)}\intL_{\RR^{n-1}}(I^{k+1-n}_1Cf)(\xx+|\yy|\vv,\vv) d\vv\\
&=(2\pi)^{2(1-n)}C(I^{k+1-n}_1Cf)(\xx,\yy).
\end{split}
\end{equation*}
and the inversion formula for $C$ follows by applying $I^{-k}_1$.
\end{proof}


%
%
\begin{rmk}
Putting $k=0$ in~\eqref{eq:inversionseries} yields
\begin{equation}\label{eq:inversionk0}
f=(2\pi)^{1-n}C I_1^{1-n}F,\qquad F=Cf.
\end{equation}
When $n$ is odd, 
equation~\eqref{eq:inversionk0} is actually equivalent to
$$
f(\xx,\yy)=(2\pi)^{1-n}(-1)^{\frac{n-1}2}\intL_{\RR^{n-1}} \triangle_\uu^{\frac{n-1}2} F(\xx+|\yy|\vv,\vv)d\vv,
$$
where $\triangle_\uu$ is the Laplacian operator with respect to $\uu$.
Thus the problem of reconstructing a function from its integrals over cones is local in odd dimensions, in the sense that computing the function at a point $(\xx,\yy)$ needs the integrals over cones passing through neighborhood of that point $(\xx,\yy)$.
On the other hand, when $n$ is even, the inversion~\eqref{eq:inversionk0} is nonlocal because the fractional Laplacian is nonlocal.
\eor\end{rmk}

Combining Proposition~\ref{thm:adjoint} and Theorem \ref{thm:inversionseries}, we have an analog of the Plancherel formula.
\begin{prop}\label{thm:plancherel}
Let $f,g\in S_r(\RRn)$ satisfy $f(\xx,\yy)=\mathbf f(\xx,|\yy|)$ and $g(\xx,\yy)=\mathbf g(\xx,|\yy|)$. For any $k<n-1$, we have
\begin{equation*}\label{eq:plancherel}
\intL_{\RR^{n-1}}\intL_{\RR^{n-1}}f(\xx,\yy)g(\xx,\yy)d\xx d\yy=(2\pi)^{1-n}\intL_{\RR^{n-1}}\intL_{\RR^{n-1}}I_1^{-k}Cf(\uu,\vv)I^{k+1-n}_1Cg(\uu,\vv)d\uu d\vv
\end{equation*}
and
\begin{equation*}\label{eq:plancherel}
\intL\half\intL_{\RR^{n-1}}\mathbf f(\xx,z)\mathbf g(\xx,z)z^{n-2}d\xx dz=(2\pi)^{1-n}\intL\half\intL_{\RR^{n-1}} I^{-k}_1\mathcal C\mathbf f(\uu,s)I^{k+1-n}_1\mathcal C\mathbf g(\uu,s)s^{n-2}d\uu ds.
\end{equation*}
\end{prop}
\begin{proof}
By Theorem~\ref{thm:inversionseries} and the Plancherel formula, we have
\begin{equation*}\label{eq:isometry}
\begin{split}
\intL_{\RR^{n-1}}\intL_{\RR^{n-1}}f(\xx,\yy)g(\xx,\yy)d\xx d\yy&\displaystyle=(2\pi)^{1-n}\intL_{\RR^{n-1}}\intL_{\RR^{n-1}}f(\xx,\yy)CI^{1-n}_1Cg(\xx,\yy)d\xx d\yy\\
&\displaystyle=(2\pi)^{1-n}\intL_{\RR^{n-1}}\intL_{\RR^{n-1}}Cf(\xx,\yy) I^{1-n}_1Cg(\xx,\yy)d\xx d\yy\\
&\displaystyle=(2\pi)^{2-2n}\intL_{\RR^{n-1}}\intL_{\RR^{n-1}}|\xxi|^k\mathcal F_1(Cf)(\xxi,\yy) |\xxi|^{-k}\mathcal F_1(I^{1-n}_1Cg)(\xxi,\yy)d\xxi d\yy.
\end{split}
\end{equation*}
Here in the second line, we used Proposition~\ref{thm:adjoint}.
\end{proof}

A completely different inversion formula for the cone transform is derived by expanding $\mathcal F_{1}f$ in spherical harmonics,
$$
\mathcal F_{1} f(\varrho\boldsymbol\varphi,\yy)=\displaystyle\sum^\infty_{l=0} \sum^{N(n-1,l)}_{k=0}(\mathcal F_{1} f)_{kl}(\varrho,\yy)Y_{lk}(\boldsymbol\varphi),$$
and
$$
\mathcal F_{} f(\varrho\boldsymbol\varphi,\boldsymbol\eta)=\displaystyle\sum^\infty_{l=0} \sum^{N(n-1,l)}_{k=0}(\mathcal F_{} f)_{kl}(\varrho,\boldsymbol\eta)Y_{lk}(\boldsymbol\varphi),
$$
where $Y_{lk}(\boldsymbol\omega)$ for $\boldsymbol\omega\in S^{n-2}$ are spherical harmonics and 
$$
N(n-1,l)=\displaystyle\frac{(2l+n-3)(n+l-4)}{l!(n-3)!},\qquad N(n-1,0)=1.
$$
Notice that
$$
\intL_{\RR^{n-1}}(\mathcal F_{1} f)_{kl}(\varrho,\yy)e^{-i\yy\cdot\boldsymbol\eta}d\yy=(\mathcal F_{} f)_{kl}(\varrho,\boldsymbol\eta).
$$
From Theorem~\ref{lem:fourierslicetheorem}, we have the following relation between $(\mathcal F_{} f)_{kl}$ and $(\mathcal F_{1} F)_{kl}$, where $F=Cf$:
\begin{equation}\label{eq:fourierfandfourier1f}
(\mathcal F_{} f)_{kl}(\varrho,\vv\varrho)=(\mathcal F_{1} F)_{kl}(\varrho,\vv).
\end{equation}

Taking the inverse Fourier transform with respect to $\vv$, we have the following theorem:
\begin{thm}
Let $f\in \mathcal S_r(\RRn)$.
If $F=Cf$, then we have 
$$
(\mathcal F_{1} f)_{kl}(\varrho,\yy)=\frac{\varrho^{n-1}}{(2\pi)^{n-1}}(\mathcal F_{} F)_{kl}(\varrho, \yy\varrho).
$$
\end{thm}
\begin{proof}
Taking the inverse Fourier transform of $(\mathcal F_{} f)_{kl}(\varrho,\vv\varrho)$ gives
\begin{equation}\label{eq:seriesofinversion}
\begin{split}
(\mathcal F_{1} f)_{kl}(\varrho,\yy)&=\frac{\varrho^{n-1}}{(2\pi)^{n-1}}\intL_{\RR^{n-1}}(\mathcal F_{} f)_{kl}(\varrho,\vv\varrho)e^{i\vv\varrho\cdot\yy}d\vv=\frac{\varrho^{n-1}}{(2\pi)^{n-1}}\intL_{\RR^{n-1}}(\mathcal F_{1} F)_{kl}(\varrho,\vv)e^{i\vv\varrho\cdot\yy}d\vv\\
&=\frac{\varrho^{n-1}}{(2\pi)^{n-1}}(\mathcal F_{} F)_{kl}(\varrho, \yy\varrho).
\end{split}
\end{equation}
where we used \eqref{eq:fourierfandfourier1f} in the second line.
\end{proof}

Because of \eqref{eq:fourierofradialfunction},~\eqref{eq:seriesofinversion} follows
\begin{equation}\label{eq:hankel}
(\mathcal F_{1} \mathbf f)_{kl}(\varrho,s)=\frac{\varrho^\frac{n+1}2s^\frac{3-n}2}{(2\pi)^{\frac{n-1}2}}\mathbf H_\frac{n-3}2(\mathcal F_{1} \mathbf F)_{kl}(\varrho,s\varrho),\quad\mathbf F=\mathcal C\mathbf f.
\end{equation}
\begin{rmk}
When $n=3$,~\eqref{eq:hankel} is already derived in~\cite{moonct14,nguyentg05}.
\eor\end{rmk}

\section{The range}\label{sec:range}
In this section, we shall determine the range of the cone transform.
\begin{thm}\label{thm:range}
The cone transform $C$ is a injection of $\mathcal S_r(\RRn)$ into $\mathcal S_{r,c}(\RRn)$,
where
$$
\mathcal S_{r,c}(\RRn)=\left\{F\in\mathcal S_{r}(\RRn):\intL_{\RR^{n-1}}F(\uu,\vv)d\uu\mbox{ is a constant for any }\vv\in\RR^{n-1}\right\}.
$$
\end{thm}
\begin{proof}
Putting $\xxi=0$ in Theorem~\ref{lem:fourierslicetheorem}, we have that $\mathcal F f(0,0)=\int_{\RR^{n-1}}Cf(\uu,\vv)d\uu$ is a constant. 
Hence it is enough to show $Cf\in \mathcal S_{r}(\RRn)$. 
To demonstrate this, we show that $\mathcal F_1(Cf)(\boldsymbol\xi,\vv)\in \mathcal S_{r}(\RRn)$. 
Also, $\mathcal F_1(Cf)(\boldsymbol\xi,\vv)$ is infinitely differentiable since $\mathcal F_{1}(Cf)(\boldsymbol\xi,\vv)=\mathcal F_{} f(\boldsymbol\xi,|\vv|\boldsymbol\xi)=\mathcal F_{}f(\boldsymbol\xi,\vv|\boldsymbol\xi|)$ and we have for any multiindexes $\alpha_1,\alpha_2,\beta_1,$ and $\beta_2$, 
\begin{equation*}
\begin{split}
&\sup\{|\boldsymbol\xi^{\alpha_1}\vv^{\alpha_2}\partial_{\boldsymbol\xi}^{\beta_1}\partial_\vv^{\beta_2} \mathcal F_{1}(Cf)(\boldsymbol\xi,\vv)|:(\boldsymbol\xi,\vv)\in\RRn\}\\
&=\sup\{|\boldsymbol\xi^{\alpha_1}\vv^{\alpha_2}|\boldsymbol\xi|^{|\beta_2|}|\vv|^{|\beta_1|}\partial_{\boldsymbol\xi}^{\beta_1}\partial_\vv^{\beta_2} \mathcal Ff(\boldsymbol\xi,\vv|\boldsymbol\xi|)|(\boldsymbol\xi,\vv)\in\RRn\}<\infty.
\end{split}
\end{equation*}
Hence, $Cf$ belongs to $\mathcal S_{r,c}(\RRn)$.
\end{proof}
When $n$ is odd, $I^{1-n}_1$ is equal to $(-1)^\frac{n-1}2\triangle^\frac{n-1}2_\uu$ and in this case, we can say more.
\begin{thm}
When $n$ is odd, the cone transform $C$ is a bijection of $\mathcal S_{r,0}(\RRn)$, where
$$
\mathcal S_{r,0}(\RRn)=\left\{f\in\mathcal S_{r}(\RRn):\intL_{\RR^{n-1}}f(\xx,\yy)d\xx=0\mbox{ for any }\yy\in\RR^{n-1}\right\}.
$$
\end{thm}
\begin{proof}
By Theorem~\ref{thm:inversionseries}, we have
$$
\mathcal F_1(Cf)(\boldsymbol\xi,\vv)=\mathcal F f(\boldsymbol\xi,\vv|\boldsymbol\xi|)=\intL_{\RR^{n-1}}\mathcal F_1f(\boldsymbol\xi,\yy)e^{-i\yy\cdot\vv|\boldsymbol\xi|}d\yy,
$$
so $\mathcal F_1(Cf)(0,\vv)$ is equal to zero.
Therefore, the range of $S_{r,0}(\RRn)$ under the cone transform $C$ is a subset of $S_{r,0}(\RRn)$.

To show $C$ is onto, let $F\in\mathcal S_{r,0}(\RRn)$.
In view of Theorem~\ref{thm:inversionseries}, it appears natural to define 
$$
f=(2\pi)^{1-n}(-1)^{\frac{n-1}2}C\triangle_\uu^\frac{n-1}2F.
$$
We know that $\triangle_\uu^\frac{n-1}2F\in \mathcal S(\RRn)$.
By Theorem~\ref{thm:range}, $f$ belongs to $\mathcal S_{r,c}(\RRn)$.
In particular, $\mathcal F_1f(\boldsymbol\xi,\yy)$ is equal to 
$$
(2\pi)^{1-n}|\boldsymbol\xi|^{n-1}\mathcal FF(\boldsymbol\xi,\yy|\boldsymbol\xi|),
$$
so for any $\yy\in\RR^{n-1},$
$$
\mathcal F_1f(0,\yy)=\intL_{\RR^{n-1}}f(\xx,\yy)d\xx=0.
$$

\end{proof}

\section{An isometry property and Sobolev space estimates}\label{sec:isometry}
In this section, we show that $I^{-\frac{n-1}2}C$ is an isometry and that the problem of reconstructing from the cone transform is well-posed in the following sense : if $f$ satisfying $Cf=F$ is uniquely determined for any $F$ belonging to a certain space, the function $f$ depends continuously on $F$. 

Let $L^2(\RR^{2(n-1)})$ be the regular $L^2$ space.
For any $\gamma\in\RR$, let the space $L^2_r(\RR^{n-1}\times\RR^{n-1})$ and $H^\gamma L^2_r(\RR^{n-1}\times\RR^{n-1})$ be defined by
\begin{equation*}
\begin{split}
&\begin{split}
L^2_r(\RR^{n-1}\times\RR^{n-1})=\{f\in L^2(\RR^{2(n-1)}):&f(\xx,\yy)=f(\xx,U\yy)\quad\mbox{ for any }(\xx,\yy)\in\RR^{n-1}\times\RR^{n-1}\\
&\mbox{and for all orthonormal transformations }U\},
\end{split}\\
&\mbox{and}\\
&H^\gamma L^2_r(\RR^{n-1}\times\RR^{n-1})=\{f\in L^2_r(\RR^{n-1}\times\RR^{n-1}):||f||_\gamma<\infty\},
\end{split}
\end{equation*}
where 
$$
||f||^2_\gamma=\intL_{\RR^{n-1}}\intL_{\RR^{n-1}}|\mathcal F_{}f(\boldsymbol\xi,\boldsymbol\eta)|^2(1+|\boldsymbol\xi|^2)^\gamma d\boldsymbol\xi d\boldsymbol\eta.
$$
Notice that $H^\gamma L^2_{\gamma}(\RR^{n-1}\times \RR^{n-1})$ is the Hilbert space with the norm $||\cdot||_{\gamma}$ and $H^0 L^2_r(\RR^{n-1}\times\RR^{n-1})=L^2_r(\RR^{n-1}\times\RR^{n-1})$.

\begin{thm}\label{thm:isometry}
The mapping $f\to I^{-\frac{n-1}2}_1 Cf$ is an isometry of $H^\gamma L^2_r(\RR^{n-1}\times\RR^{n-1})$ onto itself.
\end{thm}
\begin{proof}
We start with $||f||^2_\gamma$:
\begin{equation}\label{eq:isometry}
\begin{split}
||f||^2_\gamma&=\intL_{\RR^{n-1}}\intL_{\RR^{n-1}}|\mathcal F_{}f(\boldsymbol\xi,\vv)|^2(1+|\boldsymbol\xi|^2)^\gamma d\boldsymbol\xi d\vv\\
&=\intL_{\RR^{n-1}}\intL_{\RR^{n-1}}|\mathcal F_{1}(Cf)(\boldsymbol\xi,\vv/|\boldsymbol\xi|)|^2(1+|\boldsymbol\xi|^2)^\gamma d\boldsymbol\xi d\vv\\
&=\intL_{\RR^{n-1}}\intL_{\RR^{n-1}}|\mathcal F_{1}(Cf)(\boldsymbol\xi,\vv)|^2(1+|\boldsymbol\xi|^2)^\gamma |\boldsymbol\xi|^{n-1}d\boldsymbol\xi d\vv\\
&=(2\pi)^{1-n}||I^{-\frac{n-1}2}_1Cf||_\gamma^2.
\end{split}
\end{equation}
Here in the second and third lines, we used Theorem~\ref{lem:fourierslicetheorem} and changed variables $\vv/|\boldsymbol\xi|\to\vv$, respectively.
It remains to prove that the mapping is surjective.
It is enough to show that if $g\in H^\gamma L^2_r(\RR^{n-1}\times\RR^{n-1})$ satisfies
$$
\intL_{\RR^{n-1}}\intL_{\RR^{n-1}}\mathcal F_{1} g(\boldsymbol\xi,\vv)\mathcal F_{1}(I^{-\frac{n-1}2}_1Cf)(\boldsymbol\xi,\vv)(1+|\boldsymbol\xi|^2)^\gamma d\boldsymbol\xi d\vv=0
$$
for all $f\in \mathcal S_r(\RRn)$, then $g=0$.
Theorem~\ref{lem:fourierslicetheorem} gives us
\begin{equation*}
\begin{split}
0&=\intL_{\RR^{n-1}}\intL_{\RR^{n-1}}\mathcal F_{1} g(\boldsymbol\xi,\vv)|\boldsymbol\xi|^\frac{n-1}2\mathcal F_{1}(Cf)(\boldsymbol\xi,\vv)(1+|\boldsymbol\xi|^2)^\gamma d\boldsymbol\xi d\vv\\
&\displaystyle=\intL_{\RR^{n-1}}\intL_{\RR^{n-1}}\mathcal F_{1} g(\boldsymbol\xi,\vv)|\boldsymbol\xi|^\frac{n-1}2\mathcal F_{}f(\boldsymbol\xi,\vv|\boldsymbol\xi|)(1+|\boldsymbol\xi|^2)^\gamma d\boldsymbol\xi d\vv\\
&\displaystyle=\intL_{\RR^{n-1}}\intL_{\RR^{n-1}}\mathcal F_{1} g(\boldsymbol\xi,\vv/|\boldsymbol\xi|)|\boldsymbol\xi|^\frac{1-n}2\mathcal F_{} f(\boldsymbol\xi,\vv)(1+|\boldsymbol\xi|^2)^\gamma d\boldsymbol\xi d\vv.
\end{split}
\end{equation*}
Since $\mathcal F_{} f\in \mathcal S_r(\RR^{n-1}\times\RR^{n-1})$, 
 $\mathcal F_1 g$ is equal to zero almost everywhere, so is $g$.
\end{proof}
Notice that if $f\in L^2_r(\RR^{n-1}\times\RR^{n-1})$ with $f(\uu,\vv)=\mathbf f(\uu,|\vv|)$, then 
$$
||f||^2_0=\intL_{\RR^{n-1}}\intL_{\RR^{n-1}}|f(\uu,\vv)|^2d\uu d\vv=|S^{n-2}|\intL\half\intL_{\RR^{n-1}}|\mathbf f(\uu,z)|^2z^{n-2}d\uu dz.
$$
Let us define $H^\gamma L^2_{n-2}(\RR^{n-1}\times [0,\infty))$ by 
$$
H^\gamma L^2_{n-2}(\RR^{n-1}\times [0,\infty))=\{\mathbf f:\mathbf f(\uu,|\vv|)=f(\uu,\vv) ,\;f\in L^2_r(\RR^{n-1}\times\RR^{n-1}),\mbox{ and }||\mathbf f||_{\gamma,n-2}<\infty\},
$$
where 
$$
||\mathbf f||_{\gamma,n-2}^2=\intL_{\RR^{n-1}}\intL\half|\mathcal F_1\mathbf f(\boldsymbol\xi,z)|^2(|\boldsymbol\xi|^2+1)^\gamma z^{n-2} dzd\boldsymbol\xi.
$$
Then $H^\gamma L^2_{n-2}(\RR^{n-1}\times [0,\infty))$ is the Hilbert space with the norm $||\cdot||_{\gamma,n-2}$.
\begin{cor}
The mapping $\mathbf f\to I^{-\frac{n-1}2}_1 \mathcal C\mathbf f$ is an isometry of $H^\gamma L^2_{n-2}(\RR^{n-1}\times[0,\infty))$ onto itself.
\end{cor}
The next corollary shows the Sobolev estimates.
\begin{cor}\label{thm:stable}
For each $\gamma\in\RR$, there is a constant $c>0$ such that for $f\in \mathcal S_r(\RR^{n-1}\times\RR^{n-1})$ with $f(\xx,\yy)=\mathbf f(\xx,|\yy|)$,
$$
||f||_\gamma\leq (2\pi)^{\frac{1-n}2}||Cf||_{\gamma+\frac{n-1}2}
\quad\mbox{and}\quad
||\mathbf f||_{\gamma,n-2}\leq(2\pi)^{\frac{1-n}2}||\mathcal C\mathbf f||_{\gamma+\frac{n-1}2,n-2} .
$$
\end{cor}
This corollary  follows from~\eqref{eq:isometry}.
\begin{rmk}
When $n=3$, $ ||\mathbf f||_{\gamma,1}\leq(2\pi)^{-1}||\mathcal C\mathbf f||_{\gamma+1,1}$ was already discussed in~\cite{moonct14}.
\eor\end{rmk}

\section{The partial data}\label{sec:uniqueness}
From the inversion formula in Theorem~\ref{thm:inversionseries}, $f\in\mathcal S_r(\RRn)$ is uniquely determined by $Cf$. 
However, in many practical situations, we know only partial data, i.e., the values of $Cf$ only on a subset of its domain.
The question arises if this partial data still determines $f$ uniquely.
\begin{thm}
Let $f\in \mathcal S_r(\RRn)$.
%
The cone transform $Cf(\uu,\vv)$ is equal to zero for 
$|\uu|>\sqrt{1+|\vv|^2}$ if and only if $f(\xx,\yy)$ is equal to zero for $|\xx|^2+|\yy|^2>1$.

Also, for $\mathbf f(\xx,|\yy|)=f(\xx,\yy)$, $\mathcal C\mathbf f(\uu,s)$ is equal to zero for 
$|\uu|>\sqrt{1+s^2}$ if and only if $\mathbf f(\xx,z)$ is equal to zero for $|\xx|^2+z^2>1$ (see Figure~\ref{fig:integrationdomain}).
\end{thm}
\begin{figure}
\begin{center}
  \begin{tikzpicture}[>=stealth]
   \draw[->] (0,-1) -- (0,4) ;
   \draw[->] (-2.5,0) -- (7.5,0) ;
  \draw[thick] (3,0) -- (-1.5,4.02492) ;
    \draw[thick] (3,0) -- (7,3.57771) ;
     \draw[thick] (4,0) -- (-2,3.4641) ;
          \draw[thick] (4,0) -- (7,1.73205) ;
   \draw[blue] (2,0) arc (0:180:2cm and 2cm);

  \node at (3.5,-0.5) {$(\mathbf u,0)$};
    \node at (-0.5,-0.5) {$\mathbf 0$};
        \node at (-0.6,2.3) {$(\mathbf 0,1)$};
    \node at (7.5,-0.3) {$\xx$};
        \node at (-0.2,4.2) {$z$};

  \end{tikzpicture}
\end{center}
\caption{The upper half unit circle and $V$-shape lines}
\label{fig:integrationdomain}
\end{figure}
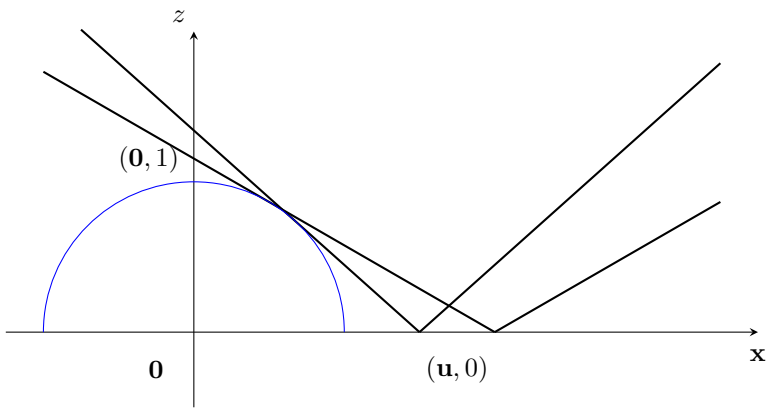
\begin{proof}
It is clear that if $f(\xx,\yy)$ is equal to zero for any $|\xx|^2+|\yy|^2>1$, then $Cf(\uu,\vv)$ is equal to zero for any $\vv\in\RR^{n-1}$ and $|\uu|>\sqrt{1+|\vv|^2}$.
From Proposition~\ref{thm:relation}, we know that for $|\boldsymbol\alpha|=|\aaa|^2(1+|\vv|^2)$ and $(\boldsymbol\alpha,\aaa,b)\in\RR^{n-1}\times\RR^{n-2}\times\RR$,
\begin{equation*}
\begin{split}
&\intL_{\RR^{n-2}} Cf(\aaa\cdot \uu'+b,\uu',\vv)d\uu'\\
&=((1+|\aaa|^2)(1+|\vv|^2))^{-1/2}Rf\left(\frac{(-1,\aaa,\boldsymbol\alpha)}{\sqrt{(1+|\aaa|^2)(1+|\vv|^2)}},\frac{b}{\sqrt{(1+|\aaa|^2)(1+|\vv|^2)}}\right).
\end{split}
\end{equation*}
By assumption, we have that
\begin{equation*}
\begin{split}
&Rf\left(\frac{(-1,\aaa,\boldsymbol\alpha)}{\sqrt{(1+|\aaa|^2)(1+|\vv|^2)}},\frac{b}{\sqrt{(1+|\aaa|^2)(1+|\vv|^2)}}\right)\\
&=\sqrt{(1+|\aaa|^2)(1+|\vv|^2)}\intL_{\RR^{n-2}} Cf(\aaa\cdot \uu'+b,\uu',\vv)d\uu'
\end{split}
\end{equation*}
is equal to zero for any $b/\sqrt{1+|\aaa|^2}>\sqrt{1+|\vv|^2}$, i.e., 
$$
\frac{b}{\sqrt{(1+|\aaa|^2)(1+|\vv|^2)}}>1
$$
because the hyperplane $\{(u_1,\uu')\in\RR\times\RR^{n-2}:u_1=\aaa\cdot \uu'+b\}$ on $\RR^{n-1}$ has the normal vector $(-1,\aaa)$ and the distance $b/\sqrt{1+|\aaa|^2}$.
%
The support theorem \cite[Theorem 3.2 in Chapter II]{natterer01} for the regular Radon transform completes our proof.
\end{proof}

We study the reconstruction problem for the limited data. 
The data $Cf(\uu,\vv)$ is known only for $\vv\in\RR^{n-1}$ with $0\leq a<|\vv|<b\leq\infty$.
In order to compute a limited reconstruction, we have to deal with the limited cone transform $C_{(a,b)}f(\uu,\vv)=\chi_{a<|\vv|<b}(\vv)Cf(\uu,\vv)$.
We define the projection operators by
$$
P_{(a,b)}f(\xx,\yy)=\mathcal F^{-1}(\chi_{a|\xxi|<|\eeta|<b|\xxi|}(\eeta)\mathcal F f(\xxi,\eeta))(\xx,\yy)
$$
and
$$
\mathbf P_{(a,b)}\mathbf f(\xx,z)=|z|^{\frac{2-n}2}\mathbf H_{\frac{n-2}2}\mathcal F^{-1}_1(\chi_{a|\xxi|<\rho<b|\xxi|}(\rho)|\rho|^{\frac{2-n}2}\mathbf H_{\frac{n-2}2}\mathcal F_1 \mathbf f(\xxi,\rho))(\xx,z).
$$ 
\begin{thm}
Let $f\in\mathcal S_r(\RRn)$.
Then we have for $k<n-1$,
$$
P_{(a,b)}I^{k}_1f=(2\pi)^{1-n}CI^{k+1-n}_1C_{(a,b)}f
$$
and 
$$
\mathbf P_{(a,b)}I^{k}_1\mathbf f=(2\pi)^{1-n}\mathcal CI^{k+1-n}_1\mathcal C_{(a,b)}\mathbf f.
$$
\end{thm}

\begin{proof}
By Theorem \ref{lem:fourierslicetheorem}, we have $\mathcal F_1(Cf)(\xxi,\vv)=\mathcal Ff(\xxi,\vv|\xxi)$, so
\begin{equation}\label{eq:slicelimited}
\begin{split}
\mathcal F(P_{(a,b)}I^{k}_1f)(\xxi,\vv|\xxi|)&=|\xxi|^{-k}\chi_{a<|\vv|<b}(\vv)\mathcal Ff(\xxi,\vv|\xxi|)=|\xxi|^{-k}\chi_{a<|\vv|<b}(\vv)\mathcal F_1(Cf)(\xxi,\vv)\\
&=|\xxi|^{-k}\mathcal F_1(C_{(a,b)}f)(\xxi,\vv).
\end{split}
\end{equation}
Similar to the proof of Theorem \ref{thm:inversionseries}, we have
$$
\begin{array}{ll}
P_{(a,b)}I^{k}_1f(\xx,\yy)&\displaystyle=(2\pi)^{2(1-n)}\intL_{\RR^{n-1}}\intL_{\RR^{n-1}}\mathcal F(P_{(a,b)}I^{k}f)(\xxi,\eeta)e^{-i(\xxi,\eeta)\cdot(\xx,\yy)} d\eeta d\xxi\\
&\displaystyle=(2\pi)^{2(1-n)}\intL_{\RR^{n-1}}\intL_{\RR^{n-1}}\mathcal F(P_{(a,b)}I^{k}f)(\xxi,\vv|\xxi|)e^{-i(\xxi,\vv|\xxi|)\cdot(\xx,\yy)}|\xxi|^{n-1}d\vv d\xxi \\
&\displaystyle=(2\pi)^{2(1-n)}\intL_{\RR^{n-1}}\intL_{\RR^{n-1}}\mathcal F_1(C_{(a,b)}f)(\xxi,\vv)e^{-i(\xxi,|\vv|\xxi)\cdot(\xx,\yy)}|\xxi|^{n-1-k}d\vv d\xxi,
\end{array}
$$
where in the second line, we changed the variables $\eeta\to\vv|\xxi|$ and in the third line, we used \eqref{eq:slicelimited} and the fact that $\mathcal F_1(C_{(a,b)}f)(\xxi,\vv)$ is radial in $\vv$.
Therefore, we have
$$
P_{(a,b)}I^{k}f(\xx,\yy)=(2\pi)^{(1-n)}\intL_{\RR^{n-1}}I^{k-n+1}_1(C_{(a,b)}f)(\xx+|\vv|\yy,\vv) d\vv.
$$
\end{proof}
\section{Conclusion}
Several types of cone transforms have been studied since the Compton camera was introduced. 
Here we study the $n$-dimensional cone transform. 
Two inversion formulas, range conditions, Sobolev space estimates, and uniqueness and reconstruction for a limited data problem are presented.

\nocite{frikelq13}
\bibliographystyle{plain}


\end{document}